\documentclass[12pt]{article}

\usepackage{amsmath,amsfonts,amssymb,amsthm,graphicx,float,color,caption,verbatim}

\usepackage{geometry}
\theoremstyle{plain}
\newtheorem{thm}{Theorem}[section]
\newtheorem{lem}[thm]{Lemma}

\newtheorem{cor}[thm]{Corollary}

\theoremstyle{definition}
\newtheorem{defn}{Definition}[section]

\newtheorem{exmp}{Example}[section]

\newtheorem{que}{Question}[section]
\newtheorem*{Ach*}{Acknowledgement}

\theoremstyle{remark}
\newtheorem{rem}{Remark}

\title{On Crossing Changes for Surface-Knots }
\author{A. Al Kharusi and T. Yashiro}
\date{}
\begin{document}
\maketitle
\begin{abstract}
In this paper, we discuss the crossing change operation along exchangeable double curves of a surface-knot diagram. 
We show that under certain condition, a finite sequence of Roseman moves preserves the property of those exchangeable double curves. As an application for this result, we also define a numerical invariant for a set of surface-knots called $du$-exchangeable set.
\end{abstract}

\section{Introduction}
A \textit{surface-knot} $F$ is an orientable connected, closed surface smoothly embedded in the Euclidean $4$-space ${\mathbb R}^4$. In particular, it is called a \textit{2-knot} if it is homeomorphic to the standard 2-sphere.  We say that two surface-knots are \textit{equivalent} if and only if they are related by an ambient isotopy of $\mathbb{R}^4$. A surface-knot is called \textit{trivial} or \textit{unknotted} if it is isotopic to the boundary of a handlebody embedded in $\mathbb{R}^3 \times \{0\}$.\\
To describe a surface-knot $F$, we consider the image of the surface-knot under the orthogonal projection $p: \mathbb{R}^4 \rightarrow \mathbb{R}^3$ that is defined by $p(x_1,x_2,x_3,x_4)=(x_1,x_2,x_3)$. We may slightly perturb $F$ by an isotopy so that its projection image in $\mathbb{R}^3$ is a \textit{generic} surface \cite{generic}; that is, its singularity set consists of at most three types: double points, isolated triple points or isolated branch points. Double points form a disjoint union of 1-manifolds that may appear as open arcs or simple closed curves. We say that such an open arc is called a \textit{double edge}. Both triple points and branch points are in the boundary of double edges of a projection. The \textit{surface-knot diagram}, or simply the \textit{diagram} of a surface-knot $F$, denoted by $D$, is the generic projection image of $F$ in 3-space with crossing informations. The \textit{triple point number} of a surface-knot is defined to be the minimal number of triple points over all possible diagrams of the surface-knot. A diagram with minimal number of triple points is called a \textit{t-minimal} diagram. A surface-knot $F$ in $\mathbb{R}^4$ is a \textit{pseudo-ribbon} if it has a surface-knot diagram with singularity set consisting of only closed 1-manifolds. \\
The crossing change in classical knot theory; that is defined by exchanging an upper arc and a lower arc at a crossing point in a knot diagram, can be generalized to theory of surface-knots. When the crossing change operation is defined, it is natural to ask the following question.
\begin{que}\mark{1.1}
Let $p(F)$ be the projection of a surface-knot $F$. Is there an unknotted surface-knot $F_0$ such that $p(F)=p(F_0)$?
\end{que}
Note that for any classical knot diagram, there exist crossings such that crossing changes at the crossings change the diagram into a diagram presenting an unknot. This fact plays an essential role to compute the invariants such as the invariants defined by skein relations, (see \cite{Homfley} for example). Moreover, the unknotting number is defined. For surface-knot diagrams, the problem has not been settled yet. In the literature on crossing changes for surface-knots, we have the following partial results. S. Kamada proved in \cite{Kamada} that a special surface braid diagram can be unknotted by crossing changes. K. Tanaka in \cite{Tanaka} showed that any pseudo-ribbon diagram can be deformed by crossing changes into a surface-knot diagram with knot group isomorphic to $\mathbb{Z}$. In this paper, we will prove that if a surface-knot $F$ has a surface-knot diagram $D$ such that crossing changes along an exchangeable union of double curves of $D$ preserve all descendants disks of $D$ and deform it into a trivial surface-knot, then there exists a sequence of surface-knot diagrams of $F$,  each of which can be unknotted by the crossing changes operation indeed.     \\
The rest of the paper is organized as follows. In section 2, we give some basics about surface-knots. In section 3,  we define a finite sequence for a surface-knot represented by its surface-knot diagrams, called a t-descendent sequence. In section 4, we review the crossing change operation and section 5 is devoted to stating and proving the main result. Finally, section 6 introduces an invariant for a set of surface-knots called $du$-exchangeable set. 


\section{Preliminaries}

\subsection{Double decker sets of surface-knot diagrams}
Let $h:\mathbb{R}^4 \rightarrow \mathbb{R}$ be the height function, $h(x_1,x_2,x_3,x_4)=x_4$. Let $\text{cl}(S)$ stand for the closure of the set $S$.\\
The closure of the set
\[
\{x \in F : \# p^{-1}(p(x))\geq 2 \}
\]
can be regarded as the image of compact 1-dimensional manifold immersed into $F$. It is divided into two families $\mathcal{S}_a=\{s_a^1,\dotso,s_a^n\}$ and $\mathcal{S}_b=\{s_b^1,\dotso,s_b^n\}$ of immersed closed intervals or simple closed curves in $F$ such that $h(x)>h(x')$ holds for any $x \in s_a^i$ and $x' \in s_b^i$ $(i=1,2,\dotso,n)$. Let $S_a=\cup_{i=1}^n\text{cl}(s_a^i)$ and let $S_b=\cup_{i=1}^n\text{cl}(s_b^i)$. The union $S_a \cup S_b$ is called the \textit{double decker set} (see \cite{generic} and \cite{lift} for details).
\subsection{Double curves of surface-knot diagrams}
Let $D$ be a surface-knot diagram of a surface-knot $F$. Let $e_1,\dotso,e_n,e_{n+1}=e_1$ be double edges and let $T_1,\dotso,T_n,T_{n+1}=T_1$ be triple points of $D$. For $i=1,\dotso,n$, assume that the boundary points of each of $e_i$ are $T_{i}$ and $T_{i+1}$ and that $e_i$ and $e_{i+1}$ are in opposition to each other at $T_{i+1}$. The closure of the union  $e_1 \cup e_2 \cup \dotso \cup e_n$ forms a circle component called a \textit{closed double curve} of the diagram. We include double point circles in the set of closed double curves.\\
Similarly, let $e_1,\dotso,e_n$ be double edges, $T_1,\dotso,T_{n-1}$ be triple points of $D$ and suppose $b_1$ and $b_n$ are branch points of $D$. Assume the boundary points of $e_1$ are the triple point $T_1$ and the branch point $b_1$. Assume also that the double edge $e_{n}$ is bounded by $T_{n-1}$ and $b_n$. For $i=2,\dotso,n-1$, the double edge $e_i$ is bounded by $T_{i-1}$ and $T_i$. If $e_i$ and $e_{i+1}$ are in opposition to each other at $T_i$ $(1,2,\dotso,n-1)$, then the closure of the union  $e_1 \cup e_2 \cup \dotso \cup e_n$ forms an arc component called an \textit{open double curve} of the diagram. We include double point open intervals in the set of open double curves. By a \textit{double curve}, we refer to an open double curve or a closed double curve. 

\subsection{Type of branches at triple points}
Let $T$ be a triple point of a surface-knot diagram $D$ and $B(T)$ a 3-ball neighbourhood of $T$. The intersection of $B(T)$ and the double edges consists of six short arcs. We call them the \textit{branches} of double edges at $T$. A branch at $T$ is called a \textit{b/m-, b/t- or m/t-branch} if it is the intersection between bottom and middle or bottom and top or middle and top sheets, respectively.

\subsection{Roseman moves}
D. Roseman introduced analogues of the Reidemeister moves as local moves to surface-knot diagrams. Let $D$ and $D'$ be surface-knot diagrams of $F$ and $F'$, respectively. It is known that $F$ and $F'$ are equivalent if and only if there exists a finite sequence of surface-knot diagrams $D=D_0\rightarrow D_1\rightarrow \dotso \rightarrow D_n=D'$ such that for all $i=0,1,\dotso,n-1,$ $D_i$ and $D_{i+1}$ differ by one of seven Roseman moves \cite{Roseman}. We will write $D \sim D'$ to indicate that $D$ and $D'$ present the same surface-knot. T. Yashiro \cite{Yashiro1} showed that the Roseman's seven moves can be described by six moves depicted in Figure 1 (see also \cite{kaw}). We call these moves also Roseman moves. Each move from left to right is denoted by $R$-$X^+$ and right to left by $R$-$X^-$ except $R$-6. The move $R$-$X^-$ is called the \textit{reverse} of $R$-$X^+$. Note that the information on height has not been specified in Figure 1.
\begin{figure}[H]
\centering
\captionsetup{font=scriptsize}       
   \mbox{\includegraphics[scale=0.55]{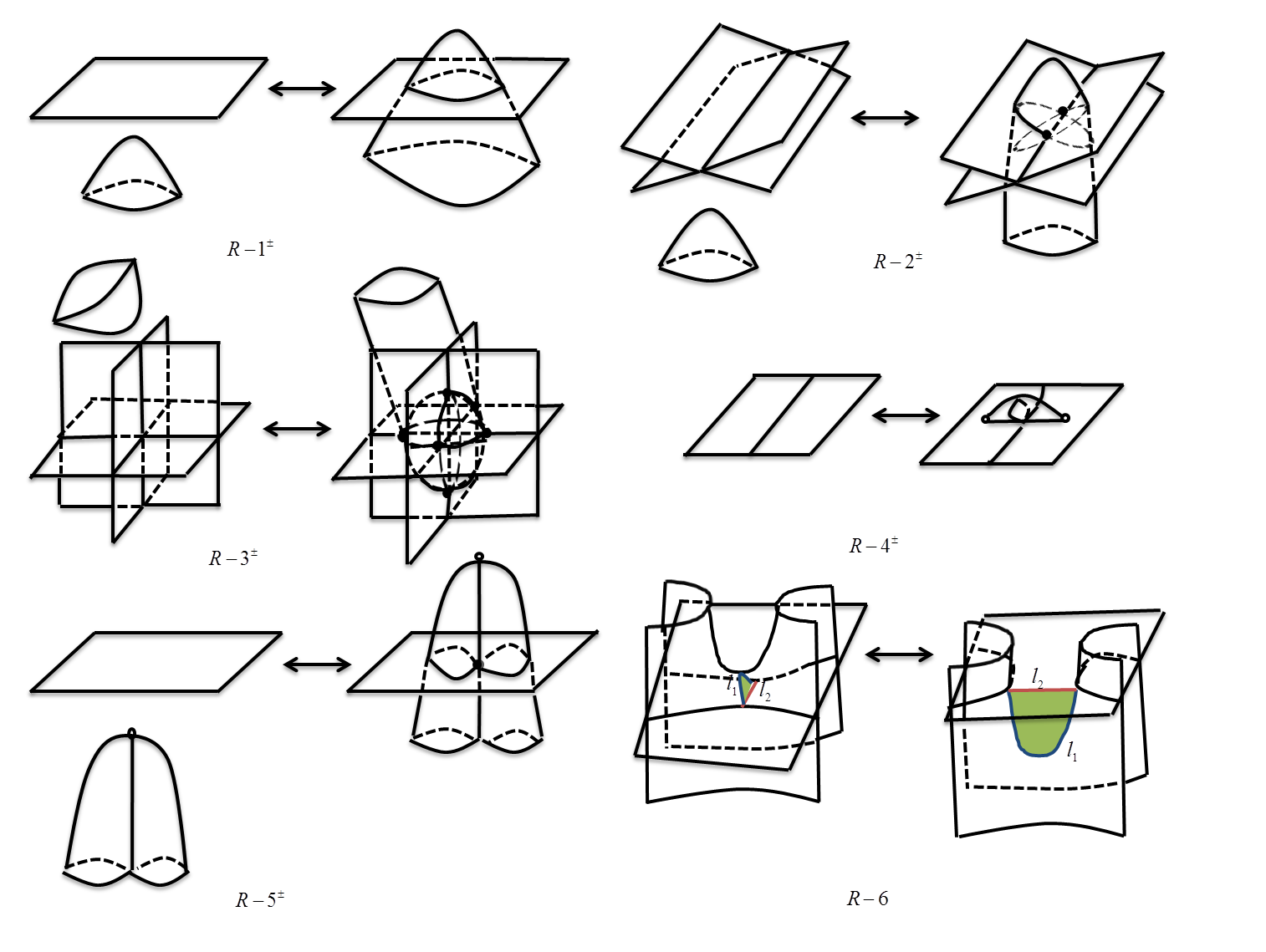}}
  \caption{}
   \end{figure}

Here we describe the $R$-$6$ move. Type $R$-$6$ move consists of deformations of two disks $P_1$ and $P_2$ in a surface-knot diagram $D$. The disk $P_1$ forms a saddle point and the other disk $P_2$ passing through the saddle point of $P_1$ in a direction perpendicular to the tangent plane at the saddle point. The two disks $P_1$ and $P_2$ intersect at two double edges, say $e_1$ from $a_1$ to $a_2$ and $e_2$ from $b_1$ to $b_2$, where $a_i$ and $b_i$ $(i=1,2)$ are boundary points of $P_2$. We give an order to these boundary points such that the four points are ordered as $\{a_1,a_2,b_1,b_2\}$ with respect to the orientation of the boundary. As $P_2$ passes through the saddle point of $P_1$, the two double edges $e_1$ and $e_2$ get closer and join at the middle point of each double edge. As a result, the new double edges $e'_1$ from $a_1$ to $b_2$ and $e'_2$ from $a_2$ to $b_1$ appear. \\

Assume that the $R$-$6$ move is applied to the pair of disks $P_1$ and $P_2$. In the notation above, let $a$ be the middle point of the double segment $e_1$ and let $b$ be the middle point of the double segment $e_2$. Then we can find a disk $P$, in a closure of one of complementary open regions of $D$, satisfying the following properties:
\begin{itemize}
\item[(1)] The interior of $P$ has empty intersection with the generic surface;
\item[(2)] $\partial P=l_1 \cup l_2$, where $l_1$ and $l_2$ are two simple arcs in $D$, each of which is terminated by $a$ and $b$;
\item[(3)] One of $l_i$ $(i=1,2)$ is on $P_1$ and the other one is on $P_2$;
\item[(4)] The pre-images of $l_i$ $(i=1,2)$ do not meet with $S_a \cup S_b$ other than their end points.
\item[(5)] The two endpoints of one of the pre-images $l_i$ $(i=1,2)$ are on $S_a$ and both endpoints of the other one are on $S_b$.

\end{itemize}
The disk $P$ is called a \textit{descendent disk} of $D$ (see the $R$-$6$ move in Figure 1) \cite{Homma}. Conversely, if a descendent disk exists, then $R$-$6$ can be applied. The two descendent disks involved in the left and right of $R$-$6$ are said to be \textit{dual to each other}. 
\section{$t$-descendent sequences}
Let $D=D_0\rightarrow D_1\rightarrow \dotso \rightarrow D_n=D'$ be a finite sequence of surface-knot diagrams satisfying the following condition:
\begin{itemize}
\item[(*)] A transition from $D_i$ to $D_{i+1}$ $(i=0,\dotso,n-1)$ is done by one of Roseman moves which can be realized by an isotopy of the surface-knot in $\mathbb{R}^4$ without creating a triple point in the projection.
\end{itemize}
In other words, the above condition says that the finite sequence of Roseman moves that connects $D$ and $D'$ does not include the Roseman moves $R$-$i^+$ $(i=2,3,5)$. Note that the reverse moves of these prohibited moves are allowed.
\begin{defn}
A finite sequence of surface-knot diagrams satisfying the condition (*) is called a \textit{t-descendent sequence}. 
\end{defn}
The terminology "t-descendent" is used to clarify that the number of triple points will not increase under the sequence. 
\begin{rem}
We point out here that M. Jabłonowski published a paper \cite{p-} in which he provided an example of two equivalent pseudo ribbon diagrams which can not be connected by a t-descendent sequence. 
\end{rem}
\section{The crossing change operation}

A crossing change operation is a local operation for a diagram of a surface-knot which has a natural analogy to the crossing changes of classical knots.

\begin{defn}
Let $D$ be a surface-knot diagram of a surface-knot. Let $\Gamma= \cup_{j=1}^r \gamma_j$ be a union of double curves in $D$. $\Gamma$ is \textit{exchangeable} if a surface-knot diagram is obtained from $D$ by changing the upper/lower information along the double curves of $\Gamma$ simultaneously. This operation is called the \textit{crossing change operation} along $\Gamma$, and we denote by $D(\Gamma)$ the surface-knot diagram obtained from $D$ by the operation. 
\end{defn}
We do not assume that $D$ and $D(\Gamma)$ present distinct surface-knots.\\

Let $T$ be a triple point of $D$. The figure below shows all possible cases of changing the crossing information around the triple point $T$, where the branches at $T$ contained in $\Gamma$ are bold lines for each possible case.
\begin{figure}[H]
\centering
\captionsetup{font=scriptsize}       
   \mbox{\includegraphics[scale=0.6]{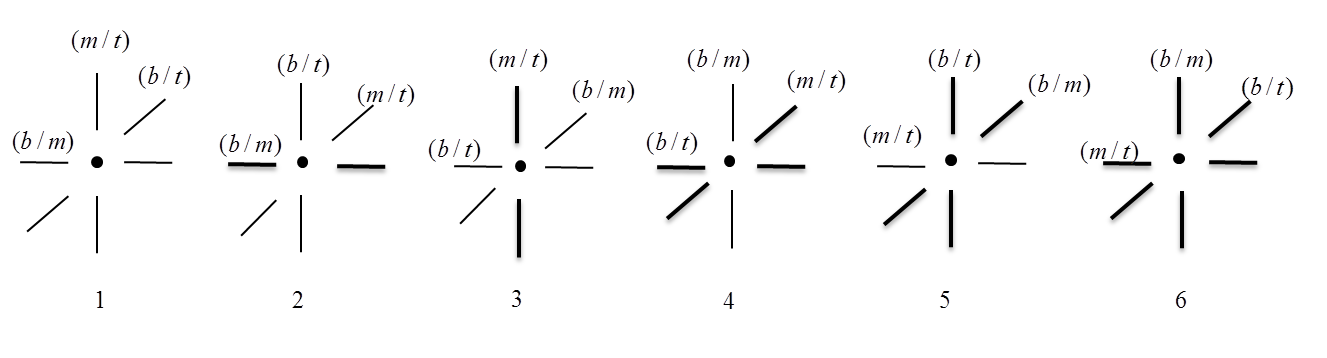}}
  \caption{}
   \end{figure}
\begin{rem}
 Let $D$ be a surface-knot diagram and $\gamma$ be a double curve in $D$. Assume $\gamma \subset \Gamma$ where $\Gamma$ is exchangeable. Suppose $\gamma$ has an edge which contains a $b/t$-branch at a triple point. Then, a $b/m$-branches or $m/t$-branches at $T$ are subsets of a double curve in $D$ that is contained in $\Gamma$.
\end{rem}
\begin{defn}
Let $\Gamma$ be an exchangeable union of double curves of a surface-knot diagram $D$. 
We say that $\Gamma$ satisfies the \textit{descendent disk condition} for $D$ if exchanging the crossing information along the double curves of $\Gamma$ preserves all descendent disks of $D$. We say that $\Gamma$ satisfies the \textit{unknotting condition} for $D$ if the surface-knot diagram $D(\Gamma)$ presents an unknotted surface-knot.
\end{defn}
\begin{defn}
A surface-knot diagram $D$ is called \textit{$du$-exchangeable} if it presents an unknotted surface-knot or it has an exchangeable union of double curves $\Gamma$ satisfying both the descendent disk condition and the unknotting condition for $D$.
\end{defn}
\begin{defn}
A surface-knot $F$ is \textit{$du$-exchangeable} if there is a surface-knot diagram $D$ presenting $F$ such that $D$ is $du$-exchangeable.
\end{defn}
\begin{rem}
It is not difficult to see that any surface-knot diagram of a surface-knot has exchangeable double curves satisfying the descendent disk condition. For example, the union of all double curves of a surface-knot diagram is exchangeable and it satisfies the descendent disk condition. 
\end{rem}
\section{The main result}
In this section we prove the main theorem in this paper (Theorem 5.4). In particular, the proof is divided into three lemmas.\\ 
Let $D$ be a surface-knot diagram of a surface-knot $F$. Let $\Gamma=\cup_{j=1}^r \gamma_j$ be an exchangeable union of double curves in $D$. The surface-knot diagram obtained after cross-change operation along $\Gamma$ applied is denoted by $D(\Gamma)$. Throughout the proof of the three lemmas in this section, $\Gamma^{(s)}$ denotes the union of double curves $\gamma_{1} \cup \gamma_{2} \cup \dotso \cup \gamma_{(s-1)} \cup \hat{\gamma}_{s} \cup \gamma_{(s+1)} \cup \dotso \cup \gamma_{r}$ that is obtained from $\Gamma$ by deleting the double curve $\gamma_{s}$. Similarly, we define $\Gamma^{(s,w)}$ to be the union of double curves that is obtained from $\Gamma$ by deleting the double curves $\gamma_s$ and $\gamma_w$. 
\begin{lem}\label{1}
Suppose that $D$ is transformed into $D'$ by one of the Roseman moves of $R$-$i^-$ $(i=2,3,5)$.   
For any exchangeable union $\Gamma$ of double curves of $D$,   
there is an exchangeable union $\Gamma'$ of double curves of $D'$ such that 
$D(\Gamma) \sim D'(\Gamma')$.  Moreover, if $\Gamma$ satisfies the descendent disk condition, then 
we may assume that $\Gamma'$ also satisfies it. 
\end{lem}
\begin{proof}
Let $\Gamma=\cup_{j=1}^r \gamma_j$ be an exchangeable union of double curves in $D$ satisfying the descendent disk condition for $D$. We need to define an exchangeable union of double curves $\Gamma'$ in $D'$ that satisfies the assertion of the lemma for each Roseman move $R$-$i^-$ $(i=2,3,5)$. \\
The Roseman move $R$-$2^-$ can be viewed that it takes the paraboloid away from the double edge so that the paraboloid does not meet with the two intersecting disks (see $R$-$2^\pm$ in Figure 1). As a result, two triple points and two closed double curves with these two triple points in each are cancelled. Let the branches of the eliminated triple points of $D$ that are the intersection of the two disks be subsets of $\gamma_s$, where $\gamma_s$ is a double curve of $D$. By applying $R$-$2^-$, $\gamma_s$ is restricted to a double curve $\gamma'_s$ of $D'$ that has less double edges by two. The exchangeable union of double curves $\Gamma'$ of $D'$ is defined as follows. 

\[
 \Gamma' =
  \begin{cases} 
      \Gamma    \hfill & \text{ if $\gamma_s \nsubseteq \Gamma$} \\
    \Gamma^{(s)}  \cup  \gamma'_{s} & \text{ if $\gamma_s \subset \Gamma$} \\
  \end{cases}
\]
We see that the diagram $D'(\Gamma')$ differs from $D(\Gamma)$ by the move $R$-$2^-$ and thus they are equivalent.\\

The move $R$-$3^-$ can be viewed that it takes the paraboloid away from the triple point $T$ so that it does not meet with the three intersecting disks  (see $R$-$3^\pm$ in Figure 1). This leads to elimination of six triple points and three closed double curves. Assume that the $b/m$- branches at $T$ are subsets of the double curve $\gamma_s$ of $D$, the $m/t$- branches at $T$ are subsets of the double curve $\gamma_w$ of $D$ and the $b/t$- branches at $T$ are subsets of the double curve $\gamma_k$ of $D$. By applying $R$-$3^-$, $\gamma_s$ is transformed to a double curve $\gamma_s'$ of $D'$ such that $\gamma_s'$ has less double edges than $\gamma_s$ by two. Similarly,  $\gamma_w$ and $\gamma_k$ of $D$ are transformed to $\gamma_w'$ and $\gamma_k'$ in $D'$, respectively. We define $\Gamma'$ by
\[
 \Gamma' =
  \begin{cases} 
       \Gamma    \hfill & \text{ if $\gamma_s, \gamma_w,\gamma_k \nsubseteq \Gamma$} \\
      \Gamma^{(s)}  \cup  \gamma'_{s} & \text{ if $\gamma_s \subset \Gamma$ only} \\
        \Gamma^{(w)}  \cup  \gamma'_{w} & \text{ if $\gamma_w \subset \Gamma$ only} \\
         \Gamma^{(k,s)}  \cup  \gamma'_{s}  \cup \gamma'_{k}& \text{ if $\gamma_s$ and $\gamma_k \subset \Gamma$ only} \\
         \Gamma^{(k,w)}  \cup  \gamma'_{w}  \cup \gamma'_{k}& \text{ if $\gamma_w$ and $\gamma_k \subset \Gamma$ only} \\
         \Gamma^{(k,s,w)}  \cup  \gamma'_{s} \cup  \gamma'_{w}   \cup \gamma'_{k}& \text{ if $\gamma_s, \gamma_w,\gamma_k \subset \Gamma$} \\
           
  \end{cases}
\]
From the definition of $\Gamma'$, we obtain that the diagram $D'(\Gamma')$ differs from $D(\Gamma)$ by the move $R$-$3^-$. It follows that they are equivalent.\\
The move $R$-$5^-$ can be described as moving a disk with a branch point away from the second disk  (see $R$-$5^\pm$ in Figure 1) and thus a triple point $T$ is cancelled. Let the branch of $T$ whose other endpoint is a branch point be a subset of a double curve $\gamma_s$ of $D$. By applying $R$-$5^-$ to $D$, $\gamma_s$ is restricted to the double curve $\gamma'_s$ of $D'$ that has less double edges by one. Define
\[
 \Gamma' =
  \begin{cases} 
      \Gamma    \hfill & \text{ if $\gamma_s \nsubseteq \Gamma$} \\
    \Gamma^{(s)}  \cup  \gamma'_{s} & \text{ if $\gamma_s \subset \Gamma$} \\
  \end{cases}
\]
It follows that the transition from $D(\Gamma)$ to $D'(\Gamma')$ is done by $R$-$5^-$. We obtain that $D(\Gamma) \sim D'(\Gamma')$.\\
It remains to prove that $\Gamma'$ defined above for each Roseman move $R$-$i^-$ $(i=2,3,5)$ satisfies the descendent disk condition. Note that none of the three moves $R$-$i^-$ $(i=2,3,5)$ create new double edges and thus no new descendent disks are involved in $D'$. Since $D$ satisfies the descendent disk condition, we can assume that $D'$ also does.          
 
\end{proof}

\begin{lem}\label{2}
Suppose that  
$D$ is transformed into $D'$ by one of the Roseman moves of $R$-$i^-$ and $R$-$i^+$  $(i=1,4)$.   
For any exchangeable union $\Gamma$ of double curves of $D$,   
there is an exchangeable union $\Gamma'$ of double curves of  $D'$ such that 
$D(\Gamma) \sim D'(\Gamma')$.  Moreover, if $\Gamma$ satisfies the descendent disk condition, then 
we may assume that $\Gamma'$ also satisfies it.
\end{lem}
\
\begin{proof}
We prove the moves $R$-$1^+$ and $R$-$1^-$. The moves  $R$-$4^+$ and $R$-$4^-$ are similarly proved. The Roseman move $R$-$1^+$ has the affect of adding a simple closed double curve, denoted by $\gamma'$ , that is independent from the other double curves of the diagram. The resulting double curve might be involved in a descendent disk where the other involved double edge is in $D$. Assume that the new descendent disk created is $P$ and that each of $\gamma'$ and $\gamma_s$ contains a boundary point of $P$, where $\gamma_s$ is a double curve of $D$. We define $\Gamma'$ in this case such that 
\[
 \Gamma' =
  \begin{cases} 
      \Gamma    \hfill & \text{ if $\gamma_s \nsubseteq \Gamma$} \\
    \Gamma \cup \gamma' & \text{ if $\gamma_s \subset \Gamma$} \\
  \end{cases}
\]

If there is no such a descendent disk, we can assume that $\Gamma'=\Gamma$. It is not hard to see that $D(\Gamma) \sim D'(\Gamma')$ in both cases and that $\Gamma'$ satisfies the descendent disk condition by the definition.
The move $R$-$1^-$ is the reverse of $R$-$1^+$. The double edge $\gamma'$ will be eliminated as a result. Denote $\gamma'$ by $\gamma_w$ for this move. Then, $\Gamma'$ of $D'$ that satisfy the assertion of the lemma can be defined such that
\[
 \Gamma' =
  \begin{cases} 
      \Gamma    \hfill & \text{ if $\gamma_w \nsubseteq \Gamma$} \\
   \Gamma^{(w)}  & \text{ if $\gamma_w \subset \Gamma$} \\
  \end{cases}
\]
The lemma follows.
\end{proof}
\begin{lem}\label{3}
Suppose that  
$D$ is transformed into $D'$ by Roseman move of $R$-$6$.   
Let $\Gamma$ be an exchangeable union of double curves of $D$ 
satisfying the descendent disk condition.  
Then there is an exchangeable union $\Gamma'$ of double curves of $D'$ such that 
$D(\Gamma) \sim D'(\Gamma')$ and that $\Gamma'$ also satisfies the descendent disk condition. 
\end{lem}
\begin{proof}
Let $P$ be a descendent disk of $D$ and suppose $R$-$6$ move is applied along $P$ to obtain the surface-knot diagram $D'$. Let $e_{1}$ and $e_{2}$ be double edges of $D$ each of which contains a boundary point of $P$. Assume that $e_1 \subset \gamma_s$ and $e_2 \subset \gamma_w$, where $\gamma_s$ and $\gamma_w$ are double curves of $D$. Since $\Gamma$ satisfies the descendent disk condition, either the upper/lower information of both $\gamma_s$ and $\gamma_w$ are exchanged or neither. In the latter case, the result follows by letting $\Gamma'=\Gamma$. On the other hand,  let $\gamma_s$ and $\gamma_w$ be subsets of $\Gamma$. Apply $R$-$6$ to obtain the surface-knot diagram $D'$. The connection between the double edges $e_1$ and $e_2$ is changed so that we obtain new double edges, say $e_1'$ and $e_2'$ in $D'$.  Assume that $e_1'$ and $e_2'$ are subsets of double curves $\gamma'_s$ and $\gamma'_w$ of $D'$, respectively. Exchanging the upper/lower information of
\[
\Gamma'= \Gamma^{(s,w)} \cup  \gamma'_s \cup \gamma'_w
\]
in $D'$ gives a surface-knot diagram equivalent to $D(\Gamma)$.
\end{proof}

\begin{thm}\label{th}
Let $F$ be a $du$-exchangeable surface-knot and $D$ be a $du$-exchangeable surface-knot diagram of $F$. Assume that $D=D_0\rightarrow D_1\rightarrow \dotso \rightarrow D_n=D'$ is a t-descendent sequence. Then, for each $i=1,\dotso,n$, $D_{i}$ is $du$-exchangeable.  
\end{thm} 
\begin{proof}
Without loss of generality, we can assume that $D'$ is obtained from $D$ by applying a single Roseman move of one of the possible types in a $t$-descendent sequence. The theorem then follows from Lemma \ref{1}, Lemma \ref{2} and Lemma \ref{3}.
\end{proof}
In the section that follows, an invariant for $du$-exchangeable surface-knots is defined.
\section{The $du$-exchange index }

Let $F$ be a $du$-exchangeable surface-knot. The $du$-exchange index $du(F)$ is defined as follows.
\begin{defn}
The \textit{$du$-exchange index $du(F)$} of a $du$-exchangeable surface-knot $F$ is the minimum number of double curves required, taken over all $du$-exchangeable diagrams representing
$F$, to convert $F$ into a trivial surface-knot.
\end{defn}
The $du$-exchange index is an invariant for $du$-exchangeable surface-knots.
\begin{exmp}
S. Satoh and A. Shima in \cite{SS} gave an estimate for a lower bound of the triple point number for tri-colourable surface-knots. Using this estimate, they showed that the 2-twist spun trefoil has the triple point number four. By following Satoh's construction of diagrams of twist-spun knots \cite{Satoh}, we obtain a surface-knot diagram of the 2-twist spun trefoil with twelve triple points. In particular, Satoh's diagram can be deformed into a t-minimal one with four triple points by a t-descendent sequence involving Roseman moves $R$-$2^-$, $R$-$5^-$ and $R$-$6$. Figure 4 depicts a schematic picture of the double curves of a t-minimal diagram of the $2$-twist spun trefoil showing the types of double branches at each triple point. 
\begin{figure}[H]
\centering
\captionsetup{font=scriptsize}      
\mbox{\includegraphics[scale=0.6]{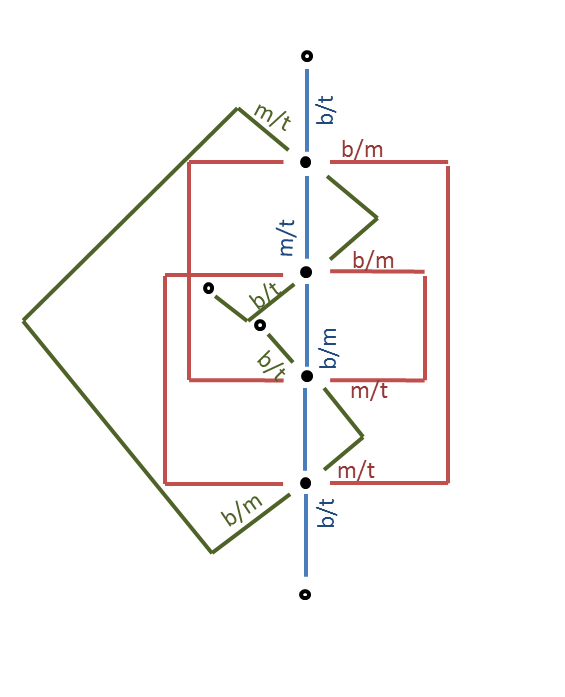}}
 \caption{Schematic picture of the double curves of a t-minimal diagram of the 2-twist spun trefoil}
\label{Fig(1)}
\end{figure} 
From the figure, we see that there are three double curves and two of them are open. Note that exchanging the crossing information of the closed double curve gives a trivial 2-knot diagram and this shows that the 2-twist spun trefoil has the $du$-exchange index equal to one.
\end{exmp}
\begin{Ach*}
We would like to express our deep sense of gratitude to Prof. Seiichi Kamada for the review of first draft of the manuscript.  We offer our most sincere appreciation and gratitude for all of his supportive comments and extensive	
efforts	for	the	thorough analysis of the manuscript. Thanks are also extended to Sultan Qaboos University for providing the funding which allowed me to undertake this research.
\end{Ach*}

Amal Al Kharusi \\
 Department of Mathematics and Statistics, College of Science, Sultan Qaboos University, Oman \\
 Email address: amalalkharusi2@gmail.com
 \\
 \\
 Tsukasa Yashiro \\
 Department of Mathematics and Statistics, College of Science, Sultan Qaboos University, Oman \\
Email address: yashiro@squ.edu.om
\end{document}